\documentclass{amsart}
\usepackage[all]{xy}
\usepackage{amssymb}
\usepackage{amscd}
\theoremstyle{plain}

\newtheorem{Thm}{Theorem}[section]

\newtheorem{Pro}[Thm]{Proposition}
\newtheorem{Cor}[Thm]{Corollary}

\theoremstyle{definition}
\newtheorem{Def}[Thm]{Definition}

\newcommand{\ZZ}{{\mathbb{Z}}}

\newcommand{\QQ}{{\mathbb{Q}}}
\newcommand{\CC}{{\mathbb{C}}}
\newcommand{\PP}{{\mathbb{P}}}

\newcommand{\KK}{{\mathcal{K}}}
\newcommand{\HH}{{\mathcal{H}}}
\newcommand{\UU}{{\mathcal{U}}}

\newcommand{\tensor}{\otimes}
\newcommand{\dirlim}{\varinjlim}
\newcommand{\invlim}{\varprojlim}
\begin{document}

%%%%%%%%%%old title
%\title[Dixmier - Douady  for Dummies]{ Dixmier - Douady  for Dummies
%}
%%%%%%%%%%%%%%%%%%%%

\title[Dixmier - Douady  for Dummies]{ The Dixmier - Douady Invariant for Dummies}
\author[]{  Claude L. Schochet}
%\footnote{ EDITOR: Please use as my identifier the following sentence:
%Claude L. Schochet is Professor of Mathematics at Wayne State University in Detroit. His email address
%is claude@math.wayne.edu }
\address{Mathematics Department,
          Wayne State University,
          Detroit MI 48202}
\email{claude@math.wayne.edu}

 \begin{abstract}
The Dixmier-Douady invariant is the primary tool in the classification
 of continuous trace $C^*$-algebras.   This expository note explores its properties from the
 perspective of classical algebraic topology.
%%%%%%%%version 2-11-09

\centerline{version 4-29-09}
\end{abstract}

\maketitle
\tableofcontents
% \vglue 2in

\begin{section}{Introduction}

The Dixmier-Douady invariant is the primary tool in the classification
 of continuous trace $C^*$-algebras. These algebras
 have come to the fore in recent years because
 of their relationship to twisted K-theory and via twisted K-theory to branes, gerbes, and string theory.

 This note sets forth
the basic properties of the Dixmier-Douady invariant
 using only classical homotopy and bundle theory.
 Algebraic topology enters the scene at once since the algebras in
 question are algebras of sections of certain fibre bundles.

 The
  results stated are all contained in the original papers of
  Dixmier and Douady \cite{DD}, Donovan and Karoubi \cite{DK}, and Rosenberg \cite{Ros}. Our treatment is novel in that
  it avoids the sheaf-theoretic techniques of the original proofs
  and substitutes   more classical algebraic topology.
Some of the proofs are borrowed directly from the recent
  paper of Atiyah and Segal \cite{AS}. Those interested in
  more detail and especially in the connections with analysis
  should consult Rosenberg \cite{Ros}, the definitive work of
  Raeburn and Williams \cite{RW}, as well as the recent paper
  of Karoubi \cite{K} and the  book by Cuntz, Meyer, and
  Rosenberg \cite{CMR}. We briefly discuss twisted $K$-theory itself, mostly in order
  to direct the interested reader to some of the (exponentially-growing) literature on the
  subject.

It is a pleasure to acknowledge the assistance of Dan Isaksen, Max Karoubi, N. C. Phillips and Jonathan Rosenberg in
the preparation of this paper.

\end{section}

\begin{section}{Fibre Bundles}

Suppose that $G$ is a topological group and $G \to T \to X$ is a principal $G$-bundle over the compact space
$X$. Then up to equivalence it is classified by a map $f$ to the classifying space $BG$ and there is a pullback diagram
\[
\begin{CD}
G @>>>   G  \\
@VVV   @VVV  \\
T  @>>>   EG  \\
@VVV   @VVV  \\
X @>f>>  BG
\end{CD}
\]
where the right column is the universal principal $G$-bundle.

Suppose
further that $F$ is some $G$-space. Then following Steenrod \cite{Steen} we may form the associated fibre bundle
\[
F \longrightarrow  T \times _G F   \longrightarrow X
\]
with fibre $F$ and structural group $G$.    Pullbacks commute with taking
associated bundles, so  there is a pullback diagram

\[
\begin{CD}
F @>>>   F  \\
@VVV   @VVV  \\
T \times _G F  @>>>   EG \times _G F  \\
@VVV   @VVV  \\
X @>f>>  BG.
\end{CD}
\]
\vglue .1in

Now suppose that $M$ is some fixed $C^*$-algebra, soon to be either the matrix ring $M_n = M_n(\CC )$
for some $n < \infty $ or the compact operators $\KK $ on some
separable Hilbert space $\HH $. Take
$G = U(M)$,
the group of unitaries of the $C^*$-algebra.  (If $M$ is not unital then we modify by first adjoining
a unit canonically to form the unital algebra $M^+$ and  then define $U(M)$ to be the kernel of the
natural homomorphism $U(M^+) \to U(M^+/M ) \cong S^1$.)  Then $U(M)$ acts naturally on $M$ by
conjugation;   denote $M$ with this action as $M^{ad}$. The center $ZU(M)$ of $U(M)$ acts trivially, and so
the action descends to an action of the {\it{projective unitary group}} $PU(M) = U(M)/ZU(M) $ on $M$,
denoted $M^{ad} $.  Note that if $M$ is simple then its center is just $\CC $ and $ZU(M) \cong S^1$ so
that $PU(M)$ is just the quotient group $(U(M))/S^1$.

Having  fixed $M$, let
\[
\zeta:\qquad   PU(M) \longrightarrow T  \to X
\]
be a principal $PU(M)$-bundle over a compact space $X$. Form the associated
fibre bundle
\[
\PP\zeta :\qquad
M^{ad}  \longrightarrow T \times _{PU(M)}  M^{ad}  \overset{p}\longrightarrow X  .
\]
This fibre bundle always has non-trivial sections. Define $A_\zeta $ to be the space of sections:
\[
A_\zeta = \Gamma (\PP\zeta ) =   \{ s: X \longrightarrow  T \times _{PU(M)}  M^{ad}  \,\,|\,\,  ps = 1 \}.
\]
This is a $C^*$-algebra with pointwise operations that are well-defined because we are using the
adjoint action.  It is unital if $M$ is unital.  If $M = M_n$ or M = $\KK $ then this is a continuous
trace $C^*$-algebra.  If $X$ is locally compact but not compact then $A_\zeta $ is still defined by using
sections that vanish at infinity and it is not unital.

Note that if $\PP\zeta $ is a trivial fibre bundle then sections correspond to functions $X \to M$ and hence
\[
A_\zeta \cong C(X)\otimes M
\]
where $C(X)$ denotes the $C^*$-algebra of continuous complex-valued functions on $X$.

 (If $X$
is only locally compact then we use $C_o $ to denote continuous functions vanishing at infinity.)

Continuous trace $C^\ast $-algebras may be defined intrinsically, of course. Here is one approach. If
$A$ is a (complex) $C^\ast $-algebra, then let $\hat A$ denote the set of unitary equivalence classes of
irreducible $*$-representations of $A$ with the Fell topology (cf. \cite{RW}).

\begin{Def} Let $X$ be a second countable locally compact Hausdorff space. A {\emph{continuous trace
$C^\ast $-algebra with $\hat A = X$ }} is a $C^\ast $-algebra $A$ with $\hat A = X$ such that the set
\[
\{ x \in A \,\,|\,\,  \text{the map} \,\,\pi \to tr(\pi(a)\pi(a)^*) \,\,\text{is finite and continuous on}\,\, \hat A \}
\]
is dense in $A$.
\end{Def}

From the definition it is easy to see that commutative $C^*$-algebras  $C_o(X)$ as well as stable
commutative   $C^*$-algebras  $C_o(X, M_n)$ and   $C_o(X , \KK )$ are continuous trace. In fact
every continuous trace algebra arises as a bundle of sections of the type we have been discussing.
\end{section}

\begin{section} {Products  }

Vector spaces come equipped with natural direct sum and tensor
product operations, and these pass over to vector bundles. Thus
if $E_1 \to X$ and $E_2 \to X$ are complex vector bundles of
dimension $r$ and $s$ respectively then we may form bundles
$E_1\oplus E_2 \to X$ of dimension $r + s$ and $E_1 \tensor E_2 \to X$
of dimension $rs$. There are two corresponding operations on
classifying spaces. The one that concerns us is the tensor
product operation. Fix some unitary isomorphism of vector spaces
\[
\CC ^r \otimes \CC ^s \cong \CC ^{rs}.
\]
 This isomorphism is unique up to homotopy, since the various
unitary groups are connected.   Let $U_n = U(M_n(\CC ))$. This determines a
homomorphism
\[
U_r \times U_s \overset{\otimes }\longrightarrow U_{rs}
\]
and associated map on classifying spaces
\[
BU_r \times BU_s \overset{\otimes }\longrightarrow BU_{rs}
\]
given by the composite
\[
BU_r \times BU_s     \,\cong\, B(U_r \times U_s)      \overset{B(\otimes )}\longrightarrow BU_{rs}  .
\]
Let $[X, Y] $ denote   homotopy classes of maps and recall that if $X$ is compact and connected then
isomorphism classes of complex $n$-plane vector bundles over $X$ correspond to elements of $[X, BU_n ]$.
Then this construction
 induces an operation
\[
[X, BU_r ] \times [X, BU_s]  \overset{\otimes }\longrightarrow
[X, BU_{rs} ].
\]
which does indeed correspond to the tensor product operation on bundles. Precisely, if
$E_1 \to X $ and $E_2 \to X$ are represented by $f_1$ and $f_2 $ respectively, then the tensor product
bundle $E_1 \otimes E_2 \to X$ is represented by $f_1 \otimes f_2 $.  (This holds at once for compact connected
spaces. If $X$ is not connected then one checks this on each component.)

The inclusion
\[
U_r \cong U_r \times \{1\} \to U_r \times U_s \to U_{rs}
\]
is denoted
\[
\alpha _{rs} : U_r \to U_{rs} .
\]

The center of $U_k$ is the group $S^1$ regarded as matrices of
the form $zI$ where $z$ is a complex number of norm $1$.
The quotient group $PU_k $ is the {\emph {projective unitary  group}}.
The fibration $S^1 \to U_k \to PU_k $ induces the sequence
\[
0 \to \ZZ  \overset{k}\longrightarrow \ZZ \to \ZZ /k  \to 0\
\]
on fundamental groups,
and, in particular, $\pi _1(PU_k) \cong \ZZ /k$.

  There is a natural induced map and commuting diagram

\[
\begin{CD}
 S^1  \times   S^1  @>>> S^1 \\\
@VVV      @VVV   \\
U_r \times U_s @>>>  U_{rs} \\
@VVV   @VVV  \\
PU_r \times PU_s @>>>  PU_{rs}
\end{CD}
\]
and this induces a tensor product operation and a commuting
diagram
\[
\begin{CD}
BU_r \times BU_s @>>> BU_{rs} \\
@VVV   @VVV  \\
BPU_r \times BPU_s @>>>  BPU_{rs}
\end{CD}
\]
It is easy to see that
\[
\pi _2(BPU_k) \cong \pi _1(PU_k) \cong \ZZ /k
\]
and that the natural map $\alpha _{rs} : PU_r \to PU_{rs}$
induces a commuting diagram
\[
\begin{CD}
\pi _2(BP_{r}) @>{(\alpha _{rs})_*}>>   \pi _2(BP_{rs}) \\
@VV\cong V     @VV\cong V  \\
\ZZ /r  @>{s}>>   \ZZ /{rs}  .
\end{CD}
\]

\vglue .1in

There is a similar structure in infinite dimensions. Fix some
separable Hilbert space $\HH $
with associated group of unitaries $\UU $ on which we impose
the strong operator topology.
The
group  $\UU $ is contractible in this topology (cf.\   \cite{RW},
Lemma 4.72).
 Fix
 some unitary isomorphism
$ \HH \otimes \HH \cong \HH $.
This is unique up to homotopy since $\UU $ is path-connected.
Then there is a canonical homomorphism
\[
\UU \times \UU  \overset{\otimes}\longrightarrow \UU
\]
and associated maps on classifying spaces
\[
\begin{CD}
B\UU  \times B\UU  @>{\otimes}>>  B\UU  \\
@VVV    @VVV  \\
BP\UU \times BP\UU  @>{\otimes}>>  BP\UU
\end{CD}
\]
where $P\UU $ denotes the infinite projective unitary group.

The action of $S^1 $ on $\UU$
 is free
and
thus
 \[
 P\UU \simeq BS^1 \simeq K(\ZZ, 2).
 \]
This implies that
\[
BP\UU \simeq K(\ZZ, 3)  .
\]

It is simpler to separate the discussion of finite and infinite
dimensional bundles at this point.

%%%%%%%%%%%%%%%%%%%%%%%%%%%%%%%%%%%%%
%%%%%%%%%%%%%%%%%%%%%%%%%%%%%%%%%%%%%%%%%%

\begin{section}{A note on cohomology for compact spaces}

If $X$ is a finite complex then the Eilenberg-Steenrod uniqueness theorem guarantees for us that
singular, simplicial, representable and \v Cech cohomology theories all coincide. Moving up to compact
spaces, one must pause to reconsider the question. The natural choice in the classical Dixmier-Douady context is
\v Cech cohomology, as this relates best to sheaf theories, and so the Dixmier-Douady invariant
was originally defined to take values in $\check H^3(X; \ZZ)$. However, a classical homotopy approach dictates
defining $H^3(X; \ZZ ) = [X, K(\ZZ, 3)]$.  Fortunately these two functors agree on compact spaces; the result
is again due to Eilenberg and Steenrod.

\begin{Pro}  \cite{ES}  On the category of compact spaces,
\v Cech cohomology is representable. That is, there is a
natural isomorphism
\[
\check{H}^n(X; \ZZ ) \,\cong \, [X, K(\ZZ , n) ].
\]
\end{Pro}

\begin{proof} The natural isomorphism is well-known for $X$ a
finite complex, by the Eilenberg-Steenrod uniqueness theorem.
Suppose that $X$ is a compact space. Then write $X = \invlim X_j$
for some inverse system of finite complexes. (See \cite{ES} chapters IX,  X, and XI  for
open covers, nerves, and inverse limits.)
Continuity of
\v Cech theory implies that
\[
\check{H}^n(X; \ZZ )  \cong \dirlim H^n(X_j ; \ZZ ) .
\]
The maps $X \to X_j$ induce natural maps
\[
[X_j , K(\ZZ, n) ] \to [X , K(\ZZ, n) ]
\]
and these coalesce to form
\[
\Phi: \dirlim [X_j , K(\ZZ, n) ]  \to [X , K(\ZZ, n) ].
\]

Claim: the map $\Phi $ is a bijection. The key fact needed
is the following result of Eilenberg - Steenrod (\cite{ES}, p. 287,
Theorem 11.9): if $ X = \invlim X_j $
is compact, $Y$ is a simplicial complex, and $f: X \to Y$
then up to homotopy $f$ factors through one of the $X_j$. This
implies immediately that $\Phi $ is onto. On the other hand, if
$g: X_j \to Y$ and the composite $X \to X_j \to Y$ is null-homotopic
then the null-homotopy factors through some $X_k \times [0,1] $ and
hence $[g] = 0$.
\end{proof}

\end{section}

\begin{section} {Bundles with fibre $\KK $}

Recall that $\KK = \KK (\HH)$ denotes the algebra of compact
operators on a separable Hilbert space $\HH $.
Let
\[
\zeta : \qquad P\UU \longrightarrow T \longrightarrow X
\]
be a principal $P\UU $-bundle with associated $C^*$-algebra $A_\zeta $.
All automorphisms of $\KK = \KK (\HH)$ are given by conjugation by
unitary operators on the Hilbert space $\HH $, so the group of
unitaries $\UU$ acts on $\KK $ by the adjoint action. The center
of the group is just $S^1 $ and it acts trivially, of course, and
so
\[
\mathrm{Aut}(\KK ) \cong \UU /S^1 = P\UU
\]
the infinite projective unitary group.
Thus
\[
[X, BP\UU ] \cong    [X, K(\ZZ , 3) ] \cong    H^3(X; \ZZ ).
\]
We may regard maps $X \to BP\UU $ as {\emph{projective vector
bundles}} in analogy with projective representations.

The resulting $C^*$-algebras $A_\zeta $ are {\emph{stable}} in the sense that
$A_\zeta \otimes \KK \cong A_\zeta $.

Define the {\emph {Dixmier-Douady}}
 invariant $\delta (A_\zeta )   $ of the
 $C^*$-algebra $A_\zeta $ to be the homotopy class of a map
 \[
f: \hat A \to BP\UU \cong K(\ZZ, 3)
\]
that classifies the bundle $E \to X$.

We note that given $A_\zeta $ then its Dixmier-Douady invariant lies
naturally in the group $H^3(\hat A_\zeta  ; \ZZ )$. The identification
of $\hat A_\zeta $ with $X$ is only given mod the group
of homeomorphisms of $X$, and hence the Dixmier-Douady invariant is
only defined modulo the action of the homeomorphism group of $X$ on $H^3(X;\ZZ )$.
Of course this action  preserves the order of the element $\delta (A_\zeta )$.

  So we have established the first parts of the following
Dixmier-Douady result:

\begin{Thm} \cite {DD}, \cite{Ros}
Let $X$ be a compact space.  Then:
\begin{enumerate}
\item There is a natural isomorphism
\[
\delta : [X, BP\UU ] \overset{\cong}\longrightarrow \check{H}^3(X, \ZZ ) .
\]
\item    Suppose given a
principal $P\UU $-bundle $\zeta $,  associated fibre bundle $\PP\zeta $,
and associated
$C^*$-algebra $A_\zeta $. Then $\delta (A_\zeta ) = 0$ if and only if $\PP\zeta $ is
equivalent to a trivial matrix bundle,  and in that case
\[
A_\zeta  \cong C(X) \tensor \KK  .
\]
\item The Dixmier-Douady invariant is additive, in the sense that
\[
\delta (A_{\zeta _1 \otimes \zeta _2}) = \delta  (A_{{\zeta _1}}) + \delta (A_{{\zeta _2}}).
\]
\item The invariant respects conjugation:
\[
\delta (A_{{\zeta ^*}}) = - \delta (A_{\zeta }).
\]
\item
Every element of $\check{H}^3(X; \ZZ )$
may be realized as the Dixmier-Douady invariant of some
infinite-dimensional bundle and associated $C^*$-algebra.
\end{enumerate}
\end{Thm}
\begin{proof} Only (3) and (4)
remain to be demonstrated. Part (3) comes down to an analysis of the
commutative diagram
\[
\begin{CD}
BP\UU \times BP\UU  @>\simeq >> K(\ZZ , 3) \times K(\ZZ , 3) \\
@VV{\otimes}V  @VVV  \\
BP\UU  @>\simeq >>   K(\ZZ, 3)
\end{CD}
\]
which deloops to
\[
\begin{CD}
P\UU \times P\UU  @>\simeq >> K(\ZZ , 2) \times K(\ZZ , 2) \\
@VV{\otimes}V  @VV{m}V  \\
P\UU  @>\simeq >>   K(\ZZ, 2).
\end{CD}
\]

The map $m$ is determined up to homotopy by its representative
in
\[
\check{H}^2( K(\ZZ , 2) \times K(\ZZ , 2) ; \ZZ ) \cong \ZZ \oplus \ZZ
\]
and this class is $(1,1)$ so that $m$ is indeed the map inducing
addition in $H^2(- ; \ZZ )$.

Part (4) is a similar argument which we omit.
\end{proof}

This result may be viewed   in a more bundle-theoretic manner. The
fibration
\[
S^1 \to \UU \to P\UU
\]
induces an exact sequence
\[
[X , K(\ZZ, 2) ] \to [X, B\UU ] \overset{\epsilon}\to [X, BP\UU ]
\overset{\delta}\to
[X, K(\ZZ, 3)]
\]
 which, after identifications, becomes
 \[
Vect _1(X)  \to Vect_\infty (X) \overset{\epsilon}\to
PVect_\infty (X) \overset{\delta}\to   H^3(X ;  \ZZ )
\]
where $Vect_k (X)$ denotes isomorphism classes of vector
bundles over $X$ of dimension $k$, $PVect_\infty (X)  $ denotes
 isomorphism classes of
projective vector bundles over $X$,
and we have identified
\[
[X , K(\ZZ, 2) ] \cong
 \check{H}^2(X; \ZZ ) \cong Vect_1(X)
 \]
using the first Chern class of the bundle.

The map $\epsilon $ takes an infinite-dimensional vector bundle $V \to X$
and associates to it the matrix bundle
\[
\epsilon (V \to X ) \,\,\,=\,\,\,   End(V) \to X
\]
where $End(V)_x = End(V_x)$,
and so if $\delta (A_\zeta ) = 0 $ then the bundle $\PP\zeta $
is isomorphic to a bundle of endomorphisms: $\PP\zeta  \cong End(V)$
as bundles. Then we use the fact that every vector bundle over
$X$ with infinite dimensional fibres is trivial as a vector
bundle (since $\UU $ is contractible), and thus there are
bundle isomorphisms
\[
\PP\zeta  \cong End(V) \cong End(X \times \HH ) \cong X \times \KK
\]
so that
\[
A_\zeta  \cong C(X) \otimes \KK
\]
as $C^*$-algebras. In fact, recalling that $\UU $ is contractible, $\UU \simeq *$,
we may extend the sequence above to read
\[
[X , K(\ZZ, 2) ] \to [X, * ] \overset{\epsilon}\to [X, BP\UU ]
\overset{\delta}\to
[X, K(\ZZ, 3] \to [X, *]
\]
and then deduce that $\delta $ is also onto.

\end{section}

 \begin{section}{Bundles with fibre $M_n(\CC )$}

The inclusion of $S^1$ as the center of $U_n$ gives rise
   to a fibration sequence
\[
S^1 \to U_n \to PU_n \to K(\ZZ , 2) \to BU_n \to BPU_n
\overset{\delta }\longrightarrow K(\ZZ, 3) .
\]

For $n \geq 2$ the map $U_n \to PU_n$ induces an isomorphism on
homotopy. Passing to classifying spaces, this yields
\[
\pi _2(BPU_n) \cong \ZZ /n
\]
as previously noted, and
\[
\pi _j(BU_n ) \cong \pi _j(BPU_n)  \qquad j > 2 .
\]
The {\emph{Dixmier-Douady invariant}} is defined to be the induced
map
\[
\delta : [X, BPU_n ] \to [X, K(\ZZ , 3)] \cong \check{H}^3(X;\ZZ ).
\]

There is a long exact sequence
\[
[X, K(\ZZ , 2 )] \to [X, BU_n] \overset{\epsilon}\longrightarrow [X, BPU_n ]
\overset{\delta}\longrightarrow [X, K(\ZZ , 3 )]
\]
which translates into
\[
\check{H}^2(X ; \ZZ ) \to Vect_n(X) \overset{\epsilon}\longrightarrow
PVect _n(X)
\overset{\delta}\longrightarrow \check{H}^3(X;\ZZ ) .
\]
  The map $\epsilon $
is defined as follows. Given a complex vector bundle $V \to X$,
then
\[
\epsilon (V \to X) \,\,=\,\,\,\,  \mathrm{End}(V) \to X
\]
where $\mathrm{End}(V) \to X $ is the endomorphism bundle of $V$; the fibre over a
point $x$ is just $\mathrm{End}(V_x)$.
This yields the third Dixmier-Douady result:

\begin{Pro}\label{endpro} Suppose that
$X$ is compact. Let
$\zeta $ be a principal $PU_n$-bundle over $X$ with associated bundle $\PP\zeta $ and
$C^*$-algebra $A_\zeta $.
 Suppose that
 $\delta (A_\zeta ) = 0$.
 Then there is a complex vector bundle   $V \to X$ of dimension $n$ over $X$
 and a bundle isomorphism
 \[
 \begin{CD}
 \PP\zeta  @>\cong >>  \mathrm{End}(V) \\
 @VVV       @VVV      \\
 X  @>{1}>>  X .
 \end{CD}
 \]
 \end{Pro}

Note that, in contrast to the infinite dimensional case, endomorphism
bundles need not be trivial bundles.   There is one improvement
 possible, and
we are indebted to Peter Gilkey for this explicit construction.

\begin{Cor} The vector bundle $V \to X$ in Proposition \ref{endpro} may be taken to have
 trivial first Chern class, so that
its structural bundle may be reduced to
an $SU_n$-bundle.
\end{Cor}

\begin{proof} Suppose that $\PP\zeta  \cong \mathrm{End}(V)$ as in the proposition. Let $L$ be
 a complex line bundle over $X$ with $c_1(L) = - c_1(V)$. Let $V' = V\tensor L$.
  Then using the fact that $L \tensor L^*$ is a trivial line bundle we have
\[
\mathrm{End}(V') \,\cong\, (V')^* \tensor V' \,\cong\,
 V\tensor V^*\tensor L \tensor L^* \,\cong\, V \tensor V^*
\,\cong\, \mathrm{End}(V)
\]
so we may replace $V$ by $V'$ and obtain  the same endomorphism bundle.
\end{proof}

Note that even though
$V$ and $V'$ have isomorphic endomorphism bundles, in general
they will not themselves be isomorphic. In fact, $End(V) \cong End(V') $
if and only if $V' \cong V\otimes L$ for some line bundle $L$.

 We can refine this observation as follows.
The diagram above expands to a natural commuting diagram
\[
\CD
SU_n @>>> PU_n @>>> K(\ZZ /n , 1) @>>> BSU_n @>>> BPU_n @>\gamma >> K(\ZZ /n , 2)  \\
@VVV  @VV{1}V  @VVV  @VVV @VV{1}V @VV\beta V   \\
  U_n @>>> PU_n @>>> K(\ZZ , 2) @>>> BU_n @>\epsilon >> BPU_n
  @>\delta >> K(\ZZ, 3)  .
\endCD
\]

The Dixmier-Douady map $\delta : BPU_n \to K(\ZZ , 3) $
factors as
\[
BPU_n \overset{\gamma}\longrightarrow K(\ZZ /n, 2) \overset{\beta}\longrightarrow K(\ZZ , 3).
\]
The map $\beta $ induces the Bockstein homomorphism
\[
\beta : \check{H}^2(X; \ZZ /n ) \to \check{H}^3(X; \ZZ )
\]
with
\[
Ker(\beta ) = \{ x \in \check{H}^2(X ; \ZZ /n ) \,\,:\,\, x \,\, {\text{lifts to an integral class in}} \,\,\check{H}^2(X ; \ZZ  )\}
\]
and
\[
Image (\beta ) = \{ x \in \check{H}^3(X ; \ZZ ) \,\,:\,\,  nx = 0 \}
\]
whose image lies in the torsion subgroup of $\check{H}^3(X; \ZZ )$.

\begin{Thm}
Let $X$ be a compact space and let $n \in \mathbb{N}$. Then:
\begin{enumerate}
\item There is a natural exact sequence
\[
0 \to
\check{H}^2(X ; \ZZ ) \overset{\sigma}\longrightarrow Vect_n(X) \overset{\epsilon}\longrightarrow [X, BPU_n ]
\overset{\delta}\longrightarrow \check{H}^3(X;\ZZ ) .
\]
\item    Suppose given a fibre
bundle $\zeta $  over the compact space $X$  with fibre $M_n $ and associated
$C^*$-algebra $A_\zeta $. Then
\begin{enumerate}
\item If $\gamma(A_\zeta ) \neq 0$ but $\delta (A_\zeta ) = 0 $ then $\gamma (A_\zeta )$ lifts to an integral class
in $\check{H}^2(X ; \ZZ ) $.
\item  If $\gamma(A_\zeta ) = 0$ then $\mathbb{P}\zeta  \cong End(V)$ with $c_1(V) = 0$.
\end{enumerate}

\item The Dixmier-Douady invariant is additive, in the sense that
\[
\delta _{\zeta _1 \otimes \zeta_2} = \delta _{\zeta_1} + \delta _{\zeta_2}.
\]
\item The invariant respects conjugation:
\[
\delta _{\zeta^*} = - \delta _\zeta.
\]
\item For any $M_n$-bundle $\zeta$, it is the case that $n\delta _\zeta = 0$.
\end{enumerate}
\end{Thm}

\begin{proof}
Most of this result has already been established.
The map $\sigma $ takes a class $c \in \check{H}^2(X; \ZZ) $
 and associates to it a vector bundle of the form $L\oplus \theta ^{n-1}$
 where $L$ is a line bundle with first Chern class $c$ and $\theta ^{n-1} $
 is a trivial bundle of dimension $n-1$. This map is one-to-one since
 it is split by the first Chern class map
\[
 c_1 : \mathrm{Vect}_n(X) \to \check{H}^2(X;\ZZ ).
 \]
Parts (3) and (4) follow as in the infinite-dimensional case.
\end{proof}

 The various maps
 $\alpha_{rs}  : PU_r \to PU_{rs} $
induce maps on classifying spaces which by abuse of language are also
denoted
 $\alpha_{rs}  : BPU_r \to BPU_{rs} $.
These maps form a directed system
 $\{ BPU_r , \alpha _{rs} \} $. Write $BPU_\infty $ for the colimit.
Note that this is \emph{not} the same as $BP\mathcal{U} = K(\ZZ , 3)$.

\begin{Pro} (Serre, \cite{GS}, pp. 228-229)

\begin{enumerate}
\item  The natural map
\[
\dirlim \pi _j(BPU_n) \to \pi _j(BPU_\infty ).
\]
is an isomorphism.
\item $\pi _2(BPU_\infty ) \cong \QQ /\ZZ $.
\item If $j \geq 2$ then $\pi _{2j}(BPU_\infty ) \cong \QQ $.
\item If  $j \geq 2$ then  $\pi_{2j - 1}(BPU_\infty ) = 0$.

\item There is a natural splitting
\[
BPU_\infty \simeq K(\QQ /\ZZ , 2) \times F
 \]
 with $\pi _j (F)  = \QQ $ for $j \geq 4 $ and even and $\pi _j (F)  = 0 $ otherwise.
\end{enumerate}
\end{Pro}

\begin{proof} Each map $\alpha _{rs} :BPU_r \to BPU_{rs}$ is a
cofibration and so (1) is immediate. We showed previously
that $\pi _2(BPU_r) \cong \ZZ /r $. The map induced by
$\alpha _{rs}$ takes the generator of $\pi _2(BPU_r)$ to
$s$ times the generator of $\pi _2(BPU_{rs})$ and so
\[
\pi _2(BPU_\infty ) = \dirlim (\ZZ /{r}, \alpha _* ) = \QQ / \ZZ .
\]
For $n >> j > 2$,   $\pi _{2j}(BPU_n)  \cong \pi _{2j}(BU_n) \cong \ZZ $
by Bott periodicity and it follows easily that
\[
\pi _{2j}(BPU_\infty) \,\,\cong\,\, \dirlim (\pi _{2j}(BPU_n) , \alpha _* ) = \QQ .
\]
Similarly, in odd degrees   homotopy groups vanish  and calculation
yields the result.
One obtains a fibration   $BPU_\infty  \to K( \QQ /\ZZ  ,2)$ ; call the fibre $F$.
Then the homotopy of $F$ is as stated and as the base space has trivial rational
cohomology.  This implies that the fibration is trivial.
 \end{proof}

The various Dixmier-Douady maps
\[
\delta  : [X, BPU_n ]  \longrightarrow \check{H}^3(X ; \ZZ )
\]
are coherent and hence pass to the limit to produce
an induced Dixmier-Douady map
\[
\delta ^\infty  : [X, BPU_\infty ]  \longrightarrow \check{H}^3(X ; \ZZ ) .
\]

It is obvious that $\delta ^\infty $ takes values in the torsion
subgroup of $\check{H}^3(X; \ZZ )$. In fact, more is true:

\begin{Pro}
Let $X$ be a compact space. Then:
\begin{enumerate}
\item
The  image of the map $\delta ^\infty $ is the whole torsion subgroup
of $\check{H}^3(X; \ZZ )$.
\item
Let $x \in \check{H}^3(X; \ZZ ) $ be a torsion class. Then there is some finite $n$ and
some principal $PU_n$-bundle $\zeta $ over $X$ such that
  $\delta (A_\zeta  ) = x$.
\end{enumerate}
\end{Pro}

\begin{proof}
The lattice of cyclic subgroups
of  $\QQ / \ZZ $
induces an equivalence
\[
\dirlim\, K(\ZZ /n , 2) \longrightarrow K(\QQ /\ZZ , 2)
\]
  Furthermore, the various Bockstein maps
$K(\ZZ /n , 2) \longrightarrow K(\ZZ, 3) $ all factor as
\[
K(\ZZ /n , 2) \longrightarrow K(\QQ /\ZZ , 2)
 \overset{\tilde\beta }\longrightarrow K(\ZZ, 3) .
\]
The exactness of the  coefficient sequence
\[
\check{H}^2(X ; \QQ /\ZZ ) \overset{\tilde\beta }\longrightarrow \check{H}^3(X ; \ZZ )
 \longrightarrow \check{H}^3(X ; \QQ )
\]
implies that the image of $ \tilde\beta  $ is exactly the
torsion subgroup of $\check{H}^3(X; \ZZ )$.
 This shows (1).

For (2), let $x$ be a torsion class. Then
  it is in the image of the Bockstein map
\[
\tilde\beta`:   \check{H}^2(X ; \QQ /\ZZ ) \to \check{H}^3(X; \ZZ )
\]
and thus may be represented as $x = \tilde\beta (y) $  where
$y \in [X, K(\QQ /\ZZ , 2]$.  The
map
\[
[X, BPU_\infty ]  \longrightarrow   [X, K(\QQ /\ZZ , 2]
\]
is onto and so the class $y$  lifts to some class
\[
z \in [X, BPU_\infty ] \cong  \dirlim [X, BPU_n ].
\]

Choose some $\zeta  \in [X, BPU_k ] $ representing $z$.
 (Note that if $x$ has order $n$ then $n$ divides
$k$ but in general $n \neq k$.)  Then
  $\delta (A_\zeta ) = x$ as required.
\end{proof}

\end{section}

\begin{section}{Twisted $K$-theory}

    Twisted $K$-theory was first introduced by Donovan and Karoubi \cite{DK} for finite-dimensional
    bundles and then by Rosenberg \cite{Ros} in the general case. In our context the point is
    to look at the $\ZZ /2 $-graded group $K_*(A_\zeta ) $.  Let $A_\zeta $ denote a continuous trace
    algebra over $X$. Recall that $K_0$ is defined for any unital
    ring as the Grothendieck group of finitely projective modules. For our purposes a topological definition
     is cleaner and so we may simply define
     \[
      K_j(A_\zeta ) \,\cong \pi _{j+1}(U(A_\zeta \otimes \KK ))   \qquad\qquad j \in \ZZ /2
      \]
      where in all cases we grade as $K_j = K^{-j} $ (and then note that by periodicity there are only two
      groups anyhow.)
 If the bundle is infinite-dimensional then it is not necessary to tensor with $\KK $ since the algebra is already stable.
  These groups are denoted in the literature by (for instance)
 \[
 K^*(X ; \zeta ) \,\,\text{or}\,\,\,   K^*(X ; \delta (\zeta ) )  \,\,\text{or}\,\,\,  K_{\delta (\zeta )}^*(X)
 \,\,\text{or}\,\,\,  K_\Delta ^*(X) .
\]
The point is that once one specifies $X$ and $\Delta = \delta (\zeta )  \in H^3(X ;\ZZ )$ then $A_\zeta $ is specified up to
equivalence and hence  $K_{\Delta }^*(X)$ makes sense as notation for $K_j(A_\zeta )$.  Here are the basic
properties:

\begin{Pro}
\begin{enumerate}
\item {\emph{Domain:\,\,}}The groups $K_{\Delta }^*(X)$ are defined for locally compact spaces $X$ and principal $PU_n$ or
$P\UU$-bundles $\zeta $ over $X$ with associated Dixmier-Douady class $\Delta = \delta (\zeta ) \in H^3(X;\ZZ)$.
\item {\emph{Naturality:\,\,}} Given $(X, \Delta  )$ together with a continuous function $f: Y \to X $ then there is induced
a map
\[
f^*:    K_{\Delta }^*(X)   \longrightarrow   K_{f^*\Delta }^*(Y)
\]
and twisted $K$-theory is natural with respect to these maps.
\item {\emph{Periodicity:\,\,}} The groups $K_{\Delta }^*(X)$  are periodic of period $2$.
\item {\emph{Product:\,\,}} There is a  cup product operation

\[
 K_{\Delta _1}^*(X)   \times K_{\Delta _2}^*(X)    \longrightarrow       K_{\Delta _1 + \Delta _2  }^*(X) .
\]
\item {\emph{Relation to untwisted $K$-theory:\,\,}} There is a natural isomorphism
\[
K_{0 }^*(X)  \cong K^*(X)
\]
where if $X$ is locally compact but not compact then $K$-theory with compact support is intended.
\end{enumerate}
\end{Pro}\qed

Karoubi notes that the cup product is not canonically defined at the level of cohomology classes. For instance, in the
finite-dimensional case, one must choose representatives from among the various algebra bundles; i.e. choose
Morita equivalences which are not canonical in general.

Rosenberg \cite{RosHom} points out the simplest case where twisted $K$-theory actually does something
interesting. Take $X = S^3$. Then the Dixmier-Douady invariant takes values in $H^3(S^3; \ZZ ) \cong \ZZ $
and hence is determined by an integer $m$. Rosenberg shows that
\[
K_m^0(S^3) = 0 \qquad\qquad   K_m^{-1}(S^3) = \ZZ /m  .
\]
He takes us further by introducing a twisted Atiyah-Hirzebruch spectral sequence (for $X$ a finite complex) converging to $K_\Delta ^*(X)$
and with
\[
E_2^* = H^*(X ; K_*(\CC ) ).
\]
Just as with the classical Atiyah-Hirzebruch spectral sequence we have $d_2 = 0$. The differential $d_3 $ is determined by
the integral Steenrod operation $Sq_\ZZ ^3 $ (as is the case classically) and (the first mention of the twist) the   class $\Delta $:
\[
d_3(x) = Sq_\ZZ ^3(x) - \Delta x.
 \]
 This spectral sequence is developed further by Atiyah and Segal \cite{AS}, \cite{AS2} who show, for instance, that
 the spectral sequence does not collapse after rationalization.

 There are other ways to twist $K$-theory and to make it equivariant, to make it real (rather than complex) or both, and the reader should consult the papers of Atiyah and Segal \cite{AS}, \cite{AS2},  Freed Hopkins and Teleman \cite{FHT}, and especially the fine survey
 paper of Karoubi \cite{K} before burying oneself in the physics literature.  That physics comes into the
 picture goes back to the observation of Witten that D-brane charges in type IIB string theory over
 a space $M$ are elements of $K_\delta ^0(M) $ - see for instance \cite{BCMMS} who show that the Dixmier-Douady
 invariant classifies bundle gerbes up to stable isomorphism and see \cite{FH}, \cite{MMS}, \cite{MR}, \cite{NT} for
 a taste of further developments.

\end{section}

\begin{section}{Rational Homotopy}

For stable continuous trace algebras, the groups $\pi _*(UA_\zeta ) $ are periodic of period 2 and
in fact correspond to the twisted $K$-theory groups. However, if $\zeta $ is a principal $PU_n$-bundle
where $n$ is finite then the natural map
\[
\pi _j(UA_\zeta )  \to K_{j-1}(A_\zeta )
\]
is neither injective nor surjective in general. Furthermore, $K$-theory obscures the geometric dimension
of the space $X$; since $K^0(S^{2n}) $ doesn't depend on $n$ and hence cannot detect it.  In this situation
a more natural question is to calculate $\pi _*(UA_\zeta )$ itself.  This is impossible even in very elementary
cases (e.g. when $A = M_2(\CC )$ ).  A more reasonable project is to calculate the
 rational homotopy groups
\[
\pi _j(UA_\zeta ) \otimes \QQ
\]
and this has been done in general for $X$ compact by \cite{LPSS}, \cite{KSS}.  The answer depends
upon the individual groups $H^j(X; \QQ )$ and upon $n$. It turns out to be independent of the principal
bundle. This is to be expected, at least after the fact, since the Dixmier-Douady invariant is finite when the
bundle is finite-dimensional and hence is trivial in the world of rational homotopy.

\end{section}

\begin{section}{Generalizations}

 {\emph{What if $X$ is assumed to be   a  $CW$-complex or, more
 generally, a compactly generated space that is
 not necessarily compact?}} First,
the definition of $A_\zeta $ doesn't lead to a $C^\ast$-algebra since infinite $CW$-complexes
are not locally compact.  These would be pro-$C^\ast$-algebras such as studied by N. C. Phillips \cite{P}.
The good news is that some of the proofs in this note generalize.
The bundle classification results require
  restriction to Dold's numerable bundles \cite{Do}, \cite{Hu}. These are bundles that are trivial
  with respect to a locally finite cover, and so one can assume, for instance, that $X$ is paracompact.
 In the infinite dimensional
case the isomorphism
\[
[X, BP\UU ] \cong [X, K(\ZZ , 3) ]
\]
is  tautological, since $BP\UU \simeq K(\ZZ, 3) $.

In the finite dimensional situation we obtain maps
\[
\mathrm{Vect}_n^{num}(X) \overset{\epsilon}\to [X, BPU_n]
\overset{\delta}\to H_{\mathrm{sing}}^3(X; \ZZ )
\]
where $\mathrm{Vect}_n^{num}(X)  $ denotes isomorphism classes
of numerable vector bundles,
and the Dixmier-Douady results still hold when singular
cohomology is understood throughout.

 {\emph{What if $X$ is assumed to be  locally compact but
 not necessarily compact?}}   In that case the definition
 of $A_\zeta $ is modified to include only those sections that
 vanish at infinity, so that the sup norm is defined and
 then $A_\zeta $ is a $C^*$-algebra again. The Dixmier-Douady
 results still hold, but it is probably better in this setting
 to shift back to the sheaf-theoretic setting of the original proofs,
 since the classification of vector bundles over locally
 compact spaces is somewhat awkward.

\end{section}

\end{section}
\end{document}